\documentclass[11pt]{amsart}
\usepackage{amssymb}
\usepackage{times}
\usepackage{url}
\usepackage{url}
\usepackage{hyperref}
\usepackage{xcolor}

\usepackage{rotating}
\usepackage{pdflscape}

\newtheorem{theorem}{Theorem}[section]
\newtheorem{lemma}[theorem]{Lemma}
\newtheorem{proposition}[theorem]{Proposition}
\newtheorem{corollary}[theorem]{Corollary}
\theoremstyle{remark}
\newtheorem{remark}[theorem]{Remark}
\newtheorem{example}[theorem]{Example}

\newcommand{\Z}{\mathbb{Z}}
\newcommand{\Q}{\mathbb{Q}}

\newcommand{\C}{\mathbb{C}}
\newcommand{\R}{\mathbb{R}}

\newcommand{\GL}{\mathrm{GL}}

\newcommand{\SL}{\mathrm{SL}}

\newcommand{\Sp}{\mathrm{Sp}}

\newcommand{\g}{\mathfrak{g}}
\newcommand{\gl}{\mathfrak{gl}}
\renewcommand{\sl}{\mathfrak{sl}}
\renewcommand{\a}{\mathfrak{a}}
\newcommand{\h}{\mathfrak{h}}

\begin{document}


\title[2-generation of simple Lie algebras and free dense subgroups of algebraic groups]{2-generation of simple Lie algebras and free dense subgroups of algebraic groups}

\author{Alla S. Detinko}
\address{School of Computer Science\\
  University of St Andrews\\
  United Kingdom}
\email{ad271@st-andrews.ac.uk}

\author{Willem A. de Graaf}
\address{Dipartimento di Matematica\\
Universit\`{a} di Trento\\
Italy}
\email{degraaf@science.unitn.it}


\begin{abstract}
We construct generating pairs of simple Lie algebras in characteristic zero. We apply this construction to exhibit infinite 
series of 2-generator Zariski dense subgroups that are free of rank 2 of
the simple algebraic groups $\SL(n, \C)$, $\Sp(n, \C)$, $G_2(\C)$.
\end{abstract}

\maketitle

\section{Introduction}\label{1}

The Tits alternative implies that each non-solvable-by-finite linear group $H$ in zero characteristic 
contains a free non-abelian subgroup \cite[p. 145]{Wehr}. Moreover free subgroups are ubiquitous in $H$ 
(see, e.g., \cite{Aoun, Epstein, FuchsRivin}). 
These subgroups are important for investigation 
of structure and properties of $H$, especially if a free subgroup is reasonably `large', 
i.e. Zariski-dense in the closure of $H$; see, for examples, \cite{Breuillard, Kuranishi}.
Despite ubiquity, explicit construction of free dense subgroups is still a problem. 
A reason is that justification of freeness of a linear group in general is
difficult (although it could be tested computationally whether a finitely generated linear group 
contains a free non-abelian subgroup, \cite{Tits}). A famous example is the problem of freeness of the group 
$$H_2(t, s) := \langle {\tiny \left( \begin{array}{cc}
1 & t \\
0 & 1\end{array}
\right), 
\left( \begin{array}{cc}
1 & 0 \\
s & 1\end{array}
\right)} \rangle.$$

Since the 1950s it is known that the group is free for $|t| \geq 2$, $|s| \geq 2$ as well as for $t$, $s$ 
algebraically independent over $\Q$; here $t, s \in \C$ (\cite[p. 168]{lyndonschupp}, \cite[p. 30-32]{Wehr}).
Nevertheless, up to now it is not known whether there are rational 
values of $t,s$ with $-2 < t, s < 2$ for which $H_2(t, s)$ is free; 
non-freeness has been justified for a range of such pairs (see, e.g., \cite{Beardon}).


In this paper we provide a method to construct free dense subgroups of rank 2 of simple algebraic groups. 
Using the method we obtain infinite series of rank 2 free dense subgroups in a number of simple algebraic groups, 
including $\SL(n, \C)$, $\Sp(n, \C)$. These free subgroups are produced 
in a similar way as $H_2(t,s)$ and are given by 
explicit generating sets of size 2. In particular, we exhibit 2-parameter families of rank 2 free dense subgroups 
of $\SL(n, \Z)$, $\Sp(n, \Z)$, i.e. explicit examples of \emph{thin} matrix groups 
(dense subgroups of infinite index \cite{Sarnak}). 
Although thin subgroups are known to be ubiquitous, explicit examples 
still are rare \cite{Fuchs, FuchsRivin, Sarnak}; see the recent paper
\cite{ab} for a series of free dense subgroups in $\SL(n,\C)$. 

Our approach is based on the construction of 2-generator sets of simple Lie
algebras which is of independent interest (cf. \cite{ionescu}, \cite[Section 2]{Kuranishi}, \cite{chisto}).
If two nilpotent elements $x,y$ of the Lie algebra $\g$ of a simple algebraic
group
$G$ generate $\g$, then the group generated by the two elements
$\exp(sx)$, $\exp(ty)$ is Zariski-dense in $G$ (see Proposition
\ref{prop:dense}).
So the density of the group is guaranteed. We also manage to show that the
groups are free for many values of $s,t$ by applying the so-called ping-pong
lemma. In order to be able to apply this lemma we need generators $x,y$ that
are sufficiently ``nice'' (i.e., seen as $n\times n$-matrices which have
very few nonzero entries, allowing for a nice description of the exponentials).
In the next section we go into the problem of finding such ``nice'' generators
of a simple Lie algebra $\g$. Then 
in Section~\ref{3} we apply it to obtain 
2-generator free dense subgroups of rank 2 of simple algebraic groups.

\section{2-generation of simple Lie algebras}\label{2}

It is known that a split semisimple Lie algebra over a field of characteristic 0
has two elements generating it (\cite{ionescu}, \cite{Kuranishi}). Here we
consider the problem to find two nilpotent elements that generate a given
semisimple Lie algebra. For our applications we want these generators to
be as ``nice'' as possible, that is, viewing the Lie algebra as a matrix
algebra, we want generators that have few nonzero coefficients, so that we
have good control over their exponentials.

First we show that it is straightforward to find two nilpotent elements
generating a given semisimple Lie algebra (Proposition \ref{prop:2}).
Subsequently we consider the Lie algebras $\sl(n,K)$, $\mathfrak{so}(n,K)$ for
$n$ odd and different from $7$ and $\mathfrak{sp}(n,K)$, where $n$ is even.
For each of these we find two ``nice'' nilpotent elements generating the Lie
algebra. Along the way we also find a ``nice'' generating set of the simple
Lie algebra of type $G_2$ in its realization as a subalgebra of $\sl(7,K)$.
In the end we prove a theorem on a subalgebra of a simple Lie algebra
generated by two specific nilpotent elements (Theorem \ref{thm:list}). This
theorem plays no role in our paper, but we note it as a curiosity that we
can prove using the previous results.

In our proofs we sometimes verify isolated cases by direct computation
in the computer algebra system {\sf GAP}4 (\cite{GAP4}).

The ground field $K$ is of characteristic 0.
We consider split semisimple Lie algebras $\g$ over $K$. Alternatively, we could
just consider semisimple Lie algebras over $\C$. Let $\h$ be a Cartan
subalgebra of $\g$. Corresponding to that $\g$ has a root system $\Phi$,
with set of simple roots $\Delta=\{\alpha_1,\ldots,\alpha_\ell\}$ and 
Cartan matrix $C$. Then
$\g$ is generated by a canonical set of generators: $h_i$, $x_{\alpha_i} \in
\g_{\alpha_i}$, $x_{-\alpha_i}\in \g_{-\alpha_i}$ for $1\leq i\leq \ell$. These
satisfy the relations (see \cite{jacobson}, \S IV.3)
\begin{align*}
  [h_i,h_j]&=0\\
  [h_i,x_{\alpha_j}] &= C(j,i)x_{\alpha_j}\\
  [h_i,x_{-\alpha_j}] &= -C(j,i)x_{-\alpha_j}\\
  [x_{\alpha_i},x_{-\alpha_j}] &= \delta_{i,j} h_i.
\end{align*}  
  
\begin{proposition}\label{prop:1}
  Let $h\in \h$ be such that $\{ \pm \alpha_i(h) \mid 1\leq i\leq \ell\}$
  consists of $2\ell$ distinct elements. Let $u=\sum_{j=1}^\ell a_i x_{\alpha_i}
  +\sum_{i=1}^\ell b_i x_{-\alpha_i}$, with all $a_i,b_i$ nonzero. Then $h,u$
  generate $\g$.
\end{proposition}

\begin{proof}
  Let $m=2\ell$ and $\beta_1,\ldots,\beta_m = \alpha_1,\ldots,\alpha_\ell,
  -\alpha_1,\ldots,-\alpha_\ell$. Then $u=\sum_j c_j x_{\beta_j}$. Define
  $[h^0,u]=u$, $[h^{i+1},u]=[h,[h^i,u]]$. Then $[h^i,u] = \sum_j \beta_j(h)^i
  c_jx_{\beta_j}$. So if we consider the elements $[h^i,u]$ for $0\leq i\leq m-1$,
  put their coefficient vectors into a matrix then we get a Vandermonde
  matrix so that the determinant is
  $$c_1\cdots c_m \prod_{i<j} (\beta_j(h)-\beta_i(h)).$$
  Hence it is nonzero, so that the space spanned by the $[h^i,u]$ contains all
  $x_{\alpha_j}$. Therefore, $h,u$ generate $\g$.
\end{proof}

\begin{proposition}\label{prop:2}
  Let $x=x_{\alpha_1}+\cdots +x_{\alpha_\ell}$, $y=b_1 x_{-\alpha_1}+\cdots +
  b_\ell x_{-\alpha_\ell}$, $b_i\neq 0$ for all $i$.
  Let $b\in K^\ell$ be the vector with coordinates $b_i$. 
  Write $v_1,\ldots,v_\ell$ for the coordinates of $Cb$. Suppose that
  $\{ \pm v_i \mid 1\leq i\leq \ell\}$ consists of $2\ell$ distinct elements.
  Then $x,y$ generate $\g$.
\end{proposition}  

\begin{proof}
  Let $h=[x,y]$ then $h=b_1h_1+\cdots +b_\ell h_\ell$ and $\alpha_i(h) = v_i$.
  So we get the conclusion from Proposition \ref{prop:1} by setting
  $u=x+y$.
\end{proof}

\begin{example}\label{exa:3}
  Let the root system be of type $A_\ell$, with the standard ordering of simple
  roots. For $1\leq i\leq \ell$ let $b_i = \sum_{j=1}^i 2^{n-j}$. 
  Let $x=x_{\alpha_1}+\cdots +x_{\alpha_\ell}$, $y=b_1 x_{-\alpha_1}+\cdots +
  b_\ell x_{-\alpha_\ell}$. Then $x,y$ generate $\g$.

  Indeed, let $v$ be as in Proposition \ref{prop:2}. Then $v_1,\ldots,v_\ell =
  2^{\ell-2},2^{\ell-3},\ldots,2,1,2^\ell$. So the result follows.
\end{example}

Let $e_{i,j}$ denote the $n\times n$-matrix with a 1 on position $(i,j)$ and
zeros elsewhere. For $x,y\in \mathfrak{gl}(n,K)$ we define $[x^0,y]=y$ and
$[x^{s+1},y]=[x,[x^s,y]]$ for $s\geq 0$. Also we use the following result,
which is well-known, and not difficult to prove. 

\begin{lemma}\label{lem:0}
  Let $\g$ be a split simple Lie algebra over $K$, with canonical generators
  $h_i$, $x_{\alpha_i}$, $x_{-\alpha_i}$, $1\leq i\leq \ell$. Let $\alpha_0$ be the
  highest root of the root system of $\g$, and $y_0\in \g_{-\alpha_0}$,
  $y_0\neq 0$. Then $\g$ is generated
  by $\{ x_{\alpha_1},\ldots,x_{\alpha_\ell},y_0\}$.
\end{lemma}

The next proposition has also been proved in \cite{chisto} for $n$ odd.

\begin{proposition}\label{lem:1}
  Let $n\geq 3$ and
  $x = \sum_{i=1}^{n-1} e_{i,i+1}$, $y=e_{n,1}$. Let $\g$ denote the Lie algebra
  generated by $x,y$. Then $\g$ is simple. If $n=2m$ is even then $\g$ is of
  type $C_m$. If $n$ is odd, then $\g$ is of type $A_{n-1}$. 
\end{proposition}

\begin{proof}
By induction on $s$ we have 
  $$[x^s,y] = \sum_{i=\max(0,s-n+1)}^{\min(s,n-1)} (-1)^i \binom{s}{i} e_{n-s+i,i+1}.$$
  Let $1\leq r\leq \lfloor \tfrac{n}{2}\rfloor$ and
  $u=[x^{2n-1-r},y]$, $v=[x^{r-1},y]$ then
  \begin{align*}
    u &= \sum_{i=n-r}^{n-1} (-1)^i \binom{2n-1-r}{i} e_{r+1-n+j,j+1}\\
    v &= \sum_{i=0}^{r-1} (-1)^i \binom{r-1}{i} e_{n-r-1-i,i+1},\\
  \end{align*}  
so after some manipulation we obtain
$$[u,v]= (-1)^{n-r} \sum_{k=1}^r  \binom{r-1}{k-1}\binom{2n-r-1}{n-r-1+k} (e_{k,k}-
    e_{n-r+k,n-r+k}).$$
  Considering this for $r=1,2,\ldots $ we see that $e_{r,r}-e_{n+1-r,n+1-r}\in\g$
  for $1\leq r\leq \lfloor \tfrac{n}{2}\rfloor$. Now
  $[e_{1,1}-e_{n,n},x]=e_{1,2}+e_{n-1,n}$ and for $r\geq 2$
  $$[e_{r,r}-e_{n+1-r,n+1-r},x] = e_{r,r+1}+e_{n-r,n-r+1}-e_{r-1,r}-e_{n+1-r,n+2-r},$$
  from which it follows that $e_{r,r+1}+e_{n-r,n-r+1}\in \g$ for $1\leq r
  <  \tfrac{n}{2}$.

  Now suppose that $n=2m$. Let $h_0 = [x^{n-1},y] = \sum_{l=1}^n (-1)^{l-1}
  \binom{n-1}{l-1} e_{l,l}$. Then $[h_0,x]=\sum_{k=1}^{n-1} (-1)^{k-1}\binom{n}{k}
  e_{k,k+1}$. But this is equal to
  $$\sum_{k=1}^{m-1} (-1)^{k-1}\binom{n}{k} (e_{k,k+1}+e_{n-k,n-k+1})+(-1)^{m-1}
  \binom{n}{m} e_{m,m+1}.$$
  So it follows that $e_{m,m+1}\in \g$.

  Define the linear map $\varphi : \gl(n,K)\to \gl(n,K)$ by
  $\varphi(e_{i,j})=(-1)^{i-j+1} e_{n-j+1,n-i+1}$. Then $\varphi$ is an automorphism
  of $\gl(n,K)$, leaving invariant $\sl(n,K)$. So its restriction to the
  latter is an automorphism of $\sl(n,K)$. By inspection it is seen that
  $\varphi$ is equal to the nontrivial diagram automorphism of $\sl(n,K)$.
  Let $\a = \{ z \in \sl(n,K) \mid \varphi(z)=z\}$. Then it is known that
  $\a$ is a simple Lie algebra of type $C_m$ (see, for example,
  \cite{kac}, \S 7.9).
  Since $x,y\in \a$ we have $\g\subset \a$. For $1\leq i\leq m-1$ set
  $x_i = e_{i,i+1}+e_{n-i,n-i+1}$, $y_i=e_{i+1,i}+e_{n-i+1,n-i}$, $h_i=e_{i,i}-
  e_{i+1,i+1}+e_{n-i,n-i}-e_{n-i+1,n-i+1}$ and $x_m=e_{m,m+1}$, $y_m=e_{m+1,m}$,
  $h_m=e_{m,m}-e_{m+1,m+1}$. Then these elements satisfy the relations of
  a canonical generating set of a Lie algebra of type $C_m$. Hence they
  generate a subalgebra of $\a$ of type $C_m$ (see for example,
  \cite{graaf_sss}). But they all lie in $\a$, so it follows that they
  generate $\a$. We have that $y$ is the lowest root vector
  of $\a$. So $\a$ is also generated by the $x_i$ along with $y$ (Lemma
  \ref{lem:0}). But all
  these elements lie in $\g$. The conclusion is that $\g=\a$.

  Now suppose that $n$ is odd. Again we have the element $h_0 = \sum_{l=1}^n
  (-1)^{l-1}\binom{n-1}{l-1} e_{l,l}$ in $\g$. Also $x_i = e_{i,i+1}+e_{n-i,
    n-i+1}\in \g$ for $1\leq i \leq \tfrac{n-1}{2}$. A short calculation
  shows that
  $$[h_0,x_i] = (-1)^i \binom{n}{i} (-e_{i,i+1}+e_{n-i,n-i+1}).$$
  Hence $e_{i,i+1}\in \g$ for $1\leq i\leq n-1$. These are the simple root
  vectors of $\sl(n,K)$. Again, $y$ is a lowest root vector of $\sl(n,K)$.
  It follows that $\g = \sl(n,K)$.
\end{proof}

\begin{proposition}\label{lem:1A}
  Let $n\geq 4$ be even, 
  $x = \sum_{i=1}^{n-1} e_{i,i+1}$, $y=e_{n-1,1}+e_{n,2}$. Let $\g$ denote the
  Lie algebra generated by $x,y$. Then $\g=\sl(n,K)$.
\end{proposition}

\begin{proof}
  For integers $s\geq 0$ and $i$ define $C(s,i) = \binom{s}{i}-\binom{s}{i-1}$.
  So $C(s,0)=1$, $C(s,s+1)=-1$ and $C(s,i)=0$ if $i<0$ or $i>s+1$. Furthermore,
  $C(s,i)+C(s,i-1) = C(s+1,i)$.

  By induction on $s$ it follows that
  $$[x^s,y] = \sum_{i=\max(0,s-n+2)}^{\min(s+1,n-1)} (-1)^i C(s,i) e_{n-s+i-1,i+1}.$$
  Set $h_0 = [x^{n-2},y]$ then
  $$h_0 = \sum_{i=0}^{n-1} (-1)^i C(n-2,i) e_{i+1,i+1}.$$
  Then $[h_0,y] = (n-2)(-e_{n-1,1}+e_{n,2})$. Hence $z=e_{n-1,1}\in \g$.
  Furthermore we set
  $$w = [x^{n-3},y] = \sum_{i=0}^{n-2} (-1)^i C(n-3,i) e_{i+2,i+1}.$$
  Then $[w,z] = C(n-3,n-2)e_{n,1}$. Hence $e_{n,1}\in \g$. So from the
  proof of Proposition \ref{lem:1} it follows that $e_{r,r+1}+e_{n-r,n-r+1}\in \g$ for
  $1\leq r<  \tfrac{n}{2}$. Using the hypothesis that $n$ is even we compute
  \begin{align*}
    [h_0,e_{r,r+1}+e_{n-r,n-r+1}]  &= (-1)^{r-1} \left( C(n-1,r)e_{r,r+1}+
    C(n-1,n-r) e_{n-r,n-r+1}\right) \\
    &= (-1)^{r-1}C(n-1,r) (e_{r,r+1}-e_{n-r,n-r+1}).\\
  \end{align*}
  It follows that $e_{r,r+1}, e_{n-r,n-r+1} \in \g$ for $1\leq r < \tfrac{n}{2}$.
  (Note that $C(n-1,r)\neq 0$ for $r < \tfrac{n}{2}$.)
  In the proof of Proposition \ref{lem:1} we have also seen that $e_{m,m+1}\in \g$,
  where $n=2m$. By Lemma \ref{lem:0} it now follows that $\g=\sl(n,K)$.
\end{proof}  

\begin{proposition}\label{lem:1B}
  Let $n=2m+1\geq 5$ be odd, 
  $x = \sum_{i=1}^{n-1} e_{i,i+1}$, $y=e_{n-1,1}+e_{n,2}$. Let $\g$ denote the
  Lie algebra generated by $x,y$. Then $\g=\mathfrak{so}(n,K)$ if $n\neq 7$
  and $\g$ is simple of type $G_2$ if $n=7$.
\end{proposition}

\begin{proof}
  The proof for $n=7$ is a direct computation with {\sf GAP}. So in the
  remainder of the proof we assume $n\neq 7$.
  
  We consider the diagram automorphism $\varphi$ as in the proof of Proposition
  \ref{lem:1}.
  Since $n$ is odd it leaves the generators $x,y$ invariant. It is known
  that $\a = \{ z \in \sl(n,K) \mid \varphi(z)=z\}$
  is a simple Lie algebra of type $B_m$ (see \cite{kac} \S 7.9). Hence
  $\a$ is $\mathfrak{so}(n,K)$ and $\g\subset  \mathfrak{so}(n,K)$.

  Also, $y$ is a root vector of $\mathfrak{so}(n,K)$ corresponding to the
  lowest root. Simple root vectors are $x_i = e_{i,i+1}+e_{n+1-i,n+2-i}$ for
  $1\leq i\leq m$. We show that $x_i\in \g$ for $1\leq i\leq m$. By Lemma
  \ref{lem:0} this finishes the proof.

  We have the same formula for $[x^s,y]$ as in the proof of the previous lemma.
  For $s=2$ that gives $[x^2,y] = e_{n-3,1}-e_{n-2,2}-e_{n-1,3}+e_{n,4}$.
  Again set $h_0 = [x^{n-2},y]$, then using the same expression for $h_0$ as
  in the proof of the previous lemma, and some manipulation,
  $$[h_0,[x^2,y]] = (C(n-2,3)-1)(e_{n-3,1}+e_{n,4}) + (C(n-2,2)-C(n-2,1))
  (e_{n-2,2}+e_{n-1,3}).$$
  If $n=5$ these coefficients are $-3$, $-2$, whereas for $n\geq 9$ they are
  both positive. 
  Together with the expression for $[x^2,y]$ this shows that
  $e_{n-3,1}+e_{n,4}$ and $e_{n-2,2}+e_{n-1,3}$ lie in $\g$. (For $n=7$ the
  coefficients are $-1$, $1$ and we see that the proof goes wrong here;
  unsurprising, as in this case the resulting Lie algebra is of type $G_2$.)

  We have $[x^{2n-4},y] = -C(2n-4,n-2)(e_{1,n-1}+e_{2,n})$ so that
  $e_{1,n-1}+e_{2,n}\in \g$. Furthermore, $[e_{1,n-1}+e_{2,n},e_{n-2,2}+e_{n-1,3}]
  =e_{1,3}-e_{n-2,n}$. It follows that $e_{1,3}-e_{n-2,n}\in \g$.
  Set $z=[x^{n-3},y]$ then
  $$z = \sum_{j=1}^{n-1} (-1)^{j-1} C(n-3,j-1) e_{j+1,j}.$$
  Now $[z,e_{1,3}-e_{n-2,n}] = e_{2,3}+e_{n-2,1}+ (n-4)(e_{1,2}+e_{n-1,n})$.

  For $0\leq r\leq m-2$ set $u=[x^{2n-4-r},y]$, $v=[x^r,y]$. Then
  \begin{align*}
    u &= \sum_{j=1}^{r+2} (-1)^{n-r-3+j} C(2n-4-r,n-r-3+j)e_{j,n-r-2+j}\\
    v &= \sum_{j=1}^{r+2} (-1)^{j-1} C(r,j-1) e_{n-r+j-2,j}.
  \end{align*}
  Then
  $$[u,v] = (-1)^{n-r} \sum_{j=1}^{r+2} C(r,j-1)C(2n-4-r,n-3-r+j) (e_{j,j}-e_{n+1-j,
    n+1-j}).$$
  Write $\gamma_{n,r}(j) = C(r,j-1)C(2n-4-r,n-3-r+j)$. Then it follows that
  $w_r =\sum_{j=1}^{r+2} \gamma_{n,r}(j)(e_{j,j}-e_{n+1-j,n+1-j}) \in\g$.
  After some manipulation it is seen that $w_0 = C(2n-4,n-2)( e_{1,1}-e_{n,n}
  +e_{2,2}-e_{n-1,n-1})$. Hence $a= e_{1,1}-e_{n,n}+e_{2,2}-e_{n-1,n-1}\in \g $.
  Now $[a,x] = e_{2,3}+e_{n-2,n-1}$. So with the above computation of
  $[z,e_{1,3}-e_{n-2,n}]$ we see that $e_{1,2}+e_{n-1,n}$ and 
  $e_{2,3}+e_{n-2,n-1}$ lie in $\g$.

  With $x_i$ as above we have
  $$[w_r,x] = \gamma_{n,r}(1) x_1 +\sum_{j=2}^{r+2} \gamma_{n,r}(j) (x_j-x_{j-1}).$$
  So $[w_r,x] = \delta_1 x_1+\cdots +\delta_{r+1} x_{r+1} +\gamma_{n,r}(r+2)
  x_{r+2}$, where the $\delta_i$ are some coefficients. We have
  $\gamma_{n,r}(r+2) = -C(2n-4-r,n-1)$. In general, for $1\leq i\leq s-1$
  we have $C(s,i)=0$ if and only if $s=2i-1$. Furthermore, $2n-4-r = 2(n-1)-1$
  is the same as $r=-1$, which is not the case. So the coefficient of
  $x_{r+2}$ is nonzero.
  As we have already shown that $x_1,x_2\in \g$ it now follows
  that $x_i\in \g$ for $1\leq i\leq m$. 
\end{proof}

\begin{theorem}\label{thm:list}
  Let $\g$ be a split simple Lie algebra over a field of characteristic 0.
  Let $x_i,y_i,h_i$ for $1\leq i\leq \ell$ be a canonical generating set of
  $\g$. Let $y$ be a lowest root vector of $\g$ and $x=\sum_{i=1}^\ell x_i$.
  Let $\a$ denote the subalgebra generated by $x,y$. Then $\a$ is simple.
  More precisely we have
  \begin{itemize}
  \item If $\g$ is of type $A_\ell$ with $\ell$ even then $\a=\g$, hence
    $\a$ is of type $A_\ell$.
  \item If $\g$ is of type $A_\ell$ with $\ell$ odd then $\a$ is of type
    $C_m$ with $m=\tfrac{\ell+1}{2}$.
  \item If $\g$ is of type $B_\ell$, $\ell\neq 3$, then $\a=\g$, hence $\a$ is
    of type $B_\ell$.
  \item If $\g$ is of type $B_3$ then $\a$ is of type $G_2$.  
  \item  If $\g$ is of type $C_\ell$ then $\a=\g$, hence $\a$ is of
    type $C_\ell$.
  \item If $\g$ is of type $D_4$ then $\a$ is of type $G_2$.
  \item If $\g$ is of type $D_\ell$ with $\ell\geq 5$ then $\a$ is of type
    $B_{\ell-1}$.
  \item If $\g$ is of type $E_6$ then $\a$ is of type $F_4$.
  \item If $\g$ if of type $E_{7,8}$ then $\a=\g$, hence $\a$ is of type
    $E_7$, $E_8$ respectively.
  \item If $\g$ is of type $F_4$ then $\a=\g$, hence $\a$ is of type $F_4$.
  \item If $\g$ is of type $G_2$ then $\a=\g$, hence $\a$ is of type $G_2$.  
  \end{itemize}  
\end{theorem}

\begin{proof}  
  The statements about the
  exceptional Lie algebras are easily verified in {\sf GAP}.

  The first, second and fifth points are shown in Proposition \ref{lem:1} for a
  particular canonical generating set. However, two different canonical
  generating sets are mapped to each other by an automorphism (\cite{jacobson},
  IV, Theorem 3). Hence these statements hold for any canonical generating set.
  The third and fourth points are shown similarly, using Proposition
  \ref{lem:1B}.

  Suppose that $\g$ is of type $D_\ell$ with $\ell\geq 5$. We consider the
  diagram automorphism $\phi$ of $\g$ which fixes $x_i,y_i,h_i$ for $1\leq i\leq
  \ell-2$, and interchanges $x_{\ell-1},x_\ell$, $y_{\ell-1},y_\ell$,
  $h_{\ell-1},h_\ell$. Then the fixed point subalgebra of $\g$ with respect to
  $\phi$ is of type $B_{\ell-1}$. Moreover, $x_1,\ldots,x_{\ell-2},x_{\ell-1}+
  x_{\ell}$, $y_1,\ldots,y_{\ell-2},y_{\ell-1}+y_{\ell}$, $h_1,\ldots,h_{\ell-2},
  h_{\ell-1}+h_{\ell}$ is a canonical generating set for this algebra.
  It follows that $x$ is exactly the sum of the positive root vectors in
  this set. Furthermore, $\phi(y)=y$ and $y$ is a lowest root vector of
  the fixed point subalgebra. Therefore by the third statement of the
  proposition it follows that $\a$ is of type $B_{\ell-1}$.

  If $\g$ is of type $D_4$ then we use a similar reasoning. But this time
  we need to take the diagram automorphism of order 3 (along with the
  computation for $G_2$). Alternatively, it can also be established
  with a straightforward computation in {\sf GAP}.
\end{proof}

\section{Free dense subgroups of simple Lie groups}\label{3}

In this section we apply generating sets of size 2 of simple Lie algebras
obtained in Section~\ref{2} to exhibit infinite series of 2-generator free
dense subgroups of rank 2 in $\SL(n, \C)$, $\Sp(n, \C)$ and in the simple
algebraic group of type $G_2$ (in its realization as a subgroup of
$\SL(7,\C)$).

As mentioned before, our subgroups are generated by two elements
$\exp(tx)$, $\exp(sy)$, where $s,t\in \C$ and $x,y$ are nilpotent elements
generating the Lie algebra of the group in question. By the following
proposition these subgroups are Zariski-dense.

\begin{proposition}\label{prop:dense}
  Let $G\subset \GL(n,\C)$ be a (connected) simple algebraic group with Lie
  algebra $\g\subset \mathfrak{gl}(n,\C)$. Let $x,y\in \g$ be nilpotent
  elements generating $\g$. Let $t',s'\in \C$ be nonzero. Then the group
  $H=\langle \exp(t'x),\exp(s'y)\rangle$ is Zariski-dense in $G$.
\end{proposition}

\begin{proof}
  Let $\overline{H}$ denote the Zariski-closure of $H$. Then $\overline{H}$ is
  an algebraic subgroup of $G$ (\cite{borel}, Proposition 1.3). Let
  $U= \{ \exp( mt'x) \mid m\in \Z\}$, then $U$ is an infinite subgroup of
  $H$. Hence $\dim \overline{U} \geq 1$. But $M=\{\exp(tx) \mid t\in \C\}$ is
  a 1-dimensional algebraic subgroup of $G$ (\cite{borel}, \S 7.3). It follows
  that $M=\overline{U}$ so that $M\subset \overline{H}$. But then the Lie
  algebra of $M$ is contained in the Lie algebra of $\overline{H}$. The
  Lie algebra of $M$ is spanned by $x$ (\cite{borel}, \S 7.3). We see that
  $x$ lies in the Lie algebra of $\overline{H}$. By the same reasoning the
  same holds for $y$. It follows that the Lie algebras of $\overline{H}$ and
  $G$ coincide. So because $G$ is connected, it follows that $\overline{H}=G$
  (\cite{borel}, \S 7.1).
\end{proof}

We use the well-known ping-pong lemma to show that a group like the group
$H$ of the previous proposition is free. For its proof we refer to
\cite{lyndonschupp}, Proposition 12.2 of Chapter III.

\begin{lemma}[Ping-Pong Lemma]\label{pingpong}
Suppose a group $G$ acts on a set $X$ and let $H_1$ and $H_2$ be subgroups of $G$.
Suppose there exist nonempty subsets $A$ and $B$ of $X$ such that
\begin{itemize}
\item $B$ is not contained in $A$,
\item $h_1A\subset B$ for each nonidentity $h_1\in H_1$,
\item $h_2B\subset A$ for each nonidentity $h_2\in H_2$.
\end{itemize}
Then the group $H=\langle H_1, H_2\rangle$ is isomorphic to the free product $H_1*H_2$
of groups $H_1$ and $H_2$. In particular, if $H_1$ and $H_2$ are infinite cyclic
groups, $H$ is isomorphic to the free nonabelian group of rank 2.
\end{lemma}

We adhere to the following notation. Let
$x := \displaystyle \sum^{n-1}_{i=1}e_{i,i+1}$, $y := e_{n,1}$, and 
$z := \displaystyle \sum_{i=1}^{n-1} b_i e_{i+1,i}$. Then $x,y$ generate 
$\mathfrak{sl}(n,\C)$ for $n$ odd and they generate
$\mathfrak{sp}(n,\C)$ when $n$ is even (Lemma \ref{lem:1}). 
Furthermore, $x,z$ generate many types
of algebras for varying choices of the $b_i$. For instance, if the $b_i$
are as in Example \ref{exa:3} then $x,z$ generate $\mathfrak{sl}(n,\C)$.

Their exponentials are
\begin{align*}
  a(t) &:= \mathrm{exp}(tx) = 1_n+\displaystyle \sum_{i < j}\frac{t^{j-i}}{(j-i)!}
  e_{i,j},\\
  b(s) &:= \mathrm{exp}(sy) = 1_n + se_{n,1},\\
  c(r) &:= \mathrm{exp}(rz) = 1_n +
  \sum_{j=2}^n \sum_{i=1}^{j-1} c_{j-i,j} \frac{r^{j-i}}{(j-i)!}e_{j,i}
  \text{ where } c_{j-i,j} = \prod_{k=1}^{j-i} b_{j-k}.\\
\end{align*}

We let these act on $\C^n$ from the left. 
This means that we view the
elements of $\C^n$ as column vectors. However, for reasons of convenience,
we write them as row vectors. Let
\begin{align*}
X_1 &:= \{(x_1, \ldots , x_n) \in \C^n |\hspace{2pt} |x_1| > |x_i|, 2 \leq i \leq n\}\\
X_2 & := \{(x_1, \ldots , x_n) \in \C^n |\hspace{2pt} |x_n| > |x_i|, 1 \leq i \leq n-1\}.\\
\end{align*}

\begin{lemma}~\label{a}
There exists $t_0 \in \R$ such that $a(t)^mX_2 \subset X_1$ for all $t\in\C$
with $|t| > t_0$, and $m \in \Z$, $m\neq 0$.
\end{lemma}
\begin{proof}
Since $a(t)^m = a(mt)$ it is enough to show that there is a $t_0>0$ such that
$a(t)X_2 \subset X_1$ when $|t| > t_0$.

Let $x := (x_1, \ldots , x_n) \in \C^n$. Then $a(t)x=(y_1,\ldots,y_n)$ with
$$y_k = \sum_{i=k}^n \frac{t^{i-k}}{(i-k)!} x_i.$$
Hence $a(t)x \in X_1$ if and only if 
$$|\sum^{n}_{i=1}x_i\frac{t^{i-1}}{(i-1)!}| > |\sum^{n}_{i=k}x_i
\frac{t^{i-k}}{(i-k)!}|
\text{ for  $2 \leq k \leq n$.} $$ 
Suppose $x \in X_2$. Then 
$$|\sum^{n}_{i=1}x_i\frac{t^{i-1}}{(i-1)!}| \geq 
|x_n|\frac{|t|^{n-1}}{(n-1)!} -  \sum^{n-1}_{i=1}|x_{i}|
\frac{|t|^{i-1}}{(i-1)!} > 
|x_n|(\frac{|t|^{n-1}}{(n-1)!} - \sum^{n-1}_{i=1}\frac{|t|^{i-1}}{(i-1)!}).$$
On the other hand, 
$$|\sum^{n}_{i=k}x_i\frac{t^{i-k}}{(i-k)!}| \leq
|x_n|\sum^{n}_{i=k}\frac{|t|^{i-k}}{(i-k)!} \leq
|x_n|\sum^{n-1}_{i=1}\frac{|t|^{i-1}}{(i-1)!}.$$
Hence if $t$ satisfies inequality
\begin{equation}\label{eqn:t}
\frac{|t|^{n-1}}{(n-1)!} - 2\sum^{n-1}_{i=1}\frac{|t|^{i-1}}{(i-1)!} > 0  
\end{equation}
then $a(t)x \in X_1$. 

Since the term with the highest exponent of $|t|$ occurs on the left with
positive coefficient, there exists a $t_0$ such that the inequalities all
hold for $|t|>t_0$.
\end{proof}

\begin{example}
Using \eqref{eqn:t} we can derive explicit values of $t_0$. 
For example, if $n = 2$ then \eqref{eqn:t} is $|t| - 2 > 0$.
Hence we can take $t_0=2$.
If $n = 3$ then \eqref{eqn:t} is $\frac{|t|^2}{2} -2(|t|+1) > 0$, i.e. we can
take $t_0 = 1+\sqrt{3}$. For $n = 4$ we have $\frac{|t|^3}{3!} -
2(\frac{|t|^2}{2} + |t|+1) > 0$, i.e. $|t|^3 -6|t|^2-12|t|-12 > 0$, 
so $t_0 = 7.8$ would do.
\end{example}

\begin{lemma}\label{b}
Set $s_0=2$. If $s\in \C$ is such that   
$|s| > s_0$ then $b(s)^mX_1 \subset X_2$ for all $m \in \Z$, $m\neq 0$.
\end{lemma}
\begin{proof}
Similar to the proof of Lemma~\ref{a}, it is enough to show that $b(s)X_1 \subset X_2$ for $|s| > 2$. 

Let $x = (x_1, \ldots , x_n) \in \C^n$. Then 
$$b(s)x = (x_1, x_2, \ldots , x_{n-1}, sx_1 + x_n).$$
Suppose $x \in X_1$, i.e. $|x_1| > |x_i|$, $2 \leq i \leq n$. 
We have $$|sx_1 + x_n| \geq |s||x_1| - |x_n| > |s||x_1| - |x_1| = |x_1|(|s|-1).$$ 
Now if $|s| > 2$ then $|x_1|(|s|-1) > |x_1| > |x_i|$, $2 \leq i \leq n-1$, i.e. $b(s)x \in X_2$, as required.
\end{proof}

\begin{lemma}~\label{c}
There is an $r_0>0$ such that for all $r\in \C$ with $|r|>r_0$ we have  
$c(r)^mX_1 \subset X_2$ for all $m \in \Z$, $m\neq 0$.
\end{lemma}

\begin{proof}
The proof is similar to the one of Lemma~\ref{a}. Let $x=(x_1,\ldots,x_n)\in
\C^n$. Then $c(r)x =(y_1,\ldots,y_n)$ with
$$y_j = \sum_{i=1}^j x_i c_{j-i,j}\frac{r^{j-i}}{(j-i)!}.$$
Suppose that $x\in X_1$. Then
\begin{align*}
|\sum_{i=1}^n x_i c_{n-i,n}\frac{r^{n-i}}{(n-i)!}| &\geq
|x_1|\frac{|r|^{n-1}}{(n-1)!}|c_{n-1,n}| -\sum_{i=2}^n |x_i| \frac{|r|^{n-i}}{(n-i)!}
|c_{n-i,n}| \\ &\geq |x_1| (\frac{|r|^{n-1}}{(n-1)!}|c_{n-1,n}| -
\sum_{i=2}^n  \frac{|r|^{n-i}}{(n-i)!}|c_{n-i,n}|).
\end{align*}
Furthermore,
$$|\sum_{i=1}^j x_i c_{j-i,j}\frac{r^{j-i}}{(j-i)!}| \leq |x_1| \sum_{i=1}^j
|c_{j-i,j}|\frac{|r|^{j-i}}{(j-i)!}.$$
So we require $r$ to satisfy the inequalities

\begin{equation}\label{eqn:r}
  \frac{|r|^{n-1}}{(n-1)!}|c_{n-1,n}| - \sum_{i=2}^n  \frac{|r|^{n-i}}{(n-i)!}
  |c_{n-i,n}| >
  \sum_{i=1}^j |c_{j-i,j}|\frac{|r|^{j-i}}{(j-i)!}
\end{equation}
for $1\leq j\leq n-1$.
Again we see that there is an $r_0$ such that these equalities are satisfied
when $r>r_0$.
\end{proof}

\begin{example}
Also here we can determine explicit values for $r_0$. Let $n=4$, and let
$b_1,b_2,b_3$ be as in Example \ref{exa:3}. Then the inequalities
\eqref{eqn:r} are $224|r|^3-84|r|^2-14|r|-2>0$, $224|r|^3-84|r|^2-22|r|-2>0$,
$224|r|^3-132|r|^2-26|r|-2>0$. The polynomials on the left-hand side each
have one real root, which is approximately, respectively, $0.53$, $0.57$,
$0.76$. Hence we could take $r_0=1$. 
\end{example}

\begin{proposition}\label{free}
Define $H_n(t,s) := \langle a(t), b(s)\rangle$. Let $t_0$ be as in Lemma \ref{a}
and $s_0$ as in Lemma \ref{b}. Then for $t,s\in\C$ with $|t| > t_0$, $|s| > s_0$
we have that $H_{2m}(t,s)$ and $H_{2m+1}(t,s)$ are free dense subgroups of rank
2 of $\Sp(2m, \C)$ and $\SL(2m+1, \C)$ respectively.
\end{proposition}

\begin{proof}
In view of Lemmas \ref{a}, \ref{b}, the ping-pong lemma (Lemma \ref{pingpong})
shows that $H_n(t,s)$ is free of rank 2. In view of Proposition \ref{lem:1},
Proposition \ref{prop:dense} shows that $H_n(t,s)$ is dense in $\Sp(n,\C)$ if
$n$ is even and in $\SL(n,\C)$ when $n$ is odd.
\end{proof}

\begin{remark}
Proposition~\ref{free} implies that $H_2(t,s)$ is free for $|t|>2, |s|>2$ (see Section~\ref{1}).
\end{remark}

We can consider the group $L_n(t,r)$ generated by $a(t)$, $c(r)$ (for various values of the
$b_i$) and use the same arguments (replacing Lemma \ref{b} with Lemma \ref{c})
to get free dense subgroups in various other algebraic groups. For example,
with the $b_i$ as in Example \ref{exa:3} we get free dense subgroups of
$\SL(n,\C)$ for every $n\geq 3$. The next example makes this precise for
$n=4$. Subsequently we have an example showing how in the same way we get a
free dense subgroup in the simple algebraic group of type $G_2$ (in its
representation as a group of $7\times 7$-matrices). 

\begin{example}~\label{exa:10}
Let $n=4$ and $b_1=8$, $b_2=12$, $b_3=14$ as in Example \ref{exa:3}.
Consider the $4\times 4$-matrices
$$\begin{pmatrix} 1 & t & \tfrac{1}{2} t^2 & \tfrac{t^3}{6}\\
  0 & 1 & t & \tfrac{1}{2} t^2 \\
  0 & 0 & 1 & t \\ 0 & 0 & 0 & 1\end{pmatrix},~~
  \begin{pmatrix} 1 & 0 & 0 & 0 \\ 8r & 1 & 0 & 0 \\ 48r^2 & 12r & 1 & 0 \\
    224r^3 & 84r^2 & 14r & 1 \end{pmatrix}.$$
  Then if $|t|>8$ and $|r|>1$ these generate a free dense subgroup of
  $\SL(4,\C)$. 
\end{example}

\begin{example}
Let $n=7$ and define the following $7\times 7$-matrices
\begin{align*}  
& x_1= e_{23}+e_{56},~ y_1 = e_{32}+e_{65},~ h_1=[x_1,y_1]\\
& x_2=e_{12}+e_{34}+e_{45}+e_{67},~ y_2=e_{21}+2e_{43}+2e_{54}+e_{76},~ h_2=[x_2,y_2].\\
\end{align*}
Let $C=\begin{pmatrix} 2 & -3 \\ -1 & 2 \end{pmatrix}$. Then $x_1,x_2,y_1,y_2,
h_1,h_2$ satisfy the relations of a canonical set of generators of the Lie
algebra of type $G_2$ (see Section \ref{2} for the definition of this concept).
Hence they generate the Lie algebra of type $G_2$. Let $v=(-1,1)$.
Then $Cv = (-5,3)$.
So by Proposition \ref{prop:2}, $x=x_1+x_2$, $z=-y_1+y_2$ generate the Lie
algebra of type $G_2$. The matrix $x$ is of the type considered in Lemma
\ref{a}. The inequality \eqref{eqn:t} becomes
$$|t|^6 -12 |t|^5 -60|t|^4-240 |t|^3-720 |t|^2-1440 |t|-1440>0.$$
The polynomial has one positive root, which is approximately $16.6$. So
we can take $t_0=17$. The matrix $z$ is of the type considered in Lemma \ref{c}.
The inequalities \eqref{eqn:r} consist of six polynomial inequalities. We
will not write down the polynomials here. We found that each of these has a
negative and a positive root. The biggest of the positive roots is
approximately $16.4$. So we can also take $r_0=17$.

The conclusion is that if $|t|,|r|>17$ then $\exp(tx)$, $\exp(rz)$ generate a
free dense subgroup of the algebraic group of type $G_2$.
\end{example}

\begin{corollary}\label{thin}
Let $n>2$ and write $n=2m$ if $n$ is even, and $n=2m+1$ if $n$ is odd.  
Let $t_0, s_0$ be as in Proposition~\ref{free}, and let
$t \in (n-1)!\Z$, $s \in \Z$ be such that $|t|>t_0$ and $|s|>s_0$.
Then $H_{2m+1}(t,s)$  and $H_{2m}(t,s)$ are thin subgroups of
$\SL(2m+1, \Z)$ and $\Sp(2m, \Z)$ respectively.
\end{corollary}
\begin{proof}
The congruence subgroup property implies that a finite index subgroup of $\SL(n, \Z)$, $\Sp(n, \Z)$, $n > 2$,
contains a principal congruence subgroup, i.e. the kernel $\Gamma_m$ of the reduction modulo $m$ homomorphism 
for some positive integer $m$. However, $\Gamma_m$ is not free as has a non-abelian unipotent (i.e. nilpotent) subgroup.
Hence, as $H_{n}(t,s)$ is free it is of infinite index.
\end{proof}

Similarly, we can construct thin subgroups of $\SL(2m, \Z)$. For example, if $t$, $r$, $b_1$, $b_2$, $b_3$
are as in Example \ref{exa:10} then $L_4(t,r) < \SL(4, \Z)$ is thin for $t \in 6\Z$, $r \in \Z$.
Notice that in \cite{SHumphries} another infinite series of free dense subgroups of $\SL(n, \Z)$ were obtained. 
Those subgroups are generated by $n$ elements and are of rank $n$. 



The above results enable us to construct further examples of free dense subgroups. 
  Consider the groups $H_n(t,s) := \langle a(t), b(s)\rangle$ as in Proposition
  \ref{free}. As we have noted, $a(t)^k = a(kt)$ and $b(s)^k=b(ks)$ for $k\in \Z$.
  Now let $S,T$ be indeterminates and consider a word
  $$W(S,T) = a(k_1T)b(l_1S)\cdots a(k_uT)b(l_uS).$$
  We say 
	that this word is nontrivial if at least one of the $k_i$, $l_j$ is
  nonzero. 
  Observe that $W(S,T)$ is an $n\times n$-matrix with entries in $\Q[S,T]$.
  As seen in
  Proposition \ref{free}, there are $s_1,t_1\in\C$ such that $H_n(s_1,t_1)$
  is free.
  This means that if $W(S,T)$ is nontrivial then $W(s_1,t_1)$ is not the
  identity
  matrix. In particular $W(S,T)$ itself is not the identity matrix.
  Now let $s_2,t_2\in \C$ be algebraically independent. Then if $W(S,T)$ is
  notrivial, $W(s_2,t_2)$ is not the identity. It follows that $H_n(s_2,t_2)$
  is free.

  With an analogous argument one sees that
  there is an $s_3\in \C$ such that $H_n(s_3,s_3)$ is free. This
  implies that $H_n(s_4,s_4)$ is free if $s_4$ is transcendental over $\Q$.
In particular, let $s_4 \in \R$ be transcendental, $|s_4| < 1$. 
Then $H_n(s_4^{k}, s_4^{k})$, $k \in \Z$, $k > 0$, gives a family of 2-generator free dense subgroups
with generators in an arbitrary small neighborhood of the identity $1_n$ (cf \cite[Theorem 7]{Kuranishi}). Now Theorem 2.1 in \cite{Breuillard}
implies that for large enough $k$ we have that
$H_n(s_4^k,s_4^k)$ is free and dense in the Euclidean topology.
Furthermore, if $s$ is algebraic and $|\sigma(s)| > max\{s_0, t_0\}$ for some $\sigma \in \rm{Gal}(\Q(s)/\Q)$
then $H_n(s, s)$ is free. 

Similarly, one can construct new families of free subgroups using $L_n(t, r)$.


\subsection*{Acknowledgments}
We are grateful to Professors Dmytro Savchuk, Vladimir Shpilrain for fruitful discussions on free groups.
We thank the Hausdorff Research Institute for Mathematics 
for hospitality and facilitation of this research through 
the Trimester Program `Logic and Algorithms in Group Theory'.
A.~S.~Detinko was supported by Marie Sk\l odowska-Curie Individual 
Fellowship grant H2020 MSCA-IF-2015, no.~704910 (EU Framework 
Programme for Research and Innovation).

\end{document}